\documentclass[11pt]{amsart}
\theoremstyle{plain}
\newtheorem{thm}{Theorem}[section]
\newtheorem{theorem}[thm]{Theorem}

\newtheorem{lemma}[thm]{Lemma}
\newtheorem{corollary}[thm]{Corollary}
\newtheorem{proposition}[thm]{Proposition}
\theoremstyle{definition}
\newtheorem{remark}[thm]{Remark}

\newtheorem{conjecture}[thm]{Conjecture}

\newtheorem{question}[thm]{Question}

\numberwithin{equation}{section}

\title [Calabi-Yau manifolds of Picard number two]{Automorphism groups of Calabi-Yau manifolds of Picard number two}

\author{Keiji Oguiso}

\address{Keiji Oguiso, Department of Mathematics, Osaka University\\
Toyonaka 560-0043 Osaka, Japan and  Korea Institute for Advanced Study, Hoegiro 87, Seoul, 130-722, Korea} \email{oguiso@math.sci.osaka-u.ac.jp}

\dedicatory{Dedicated to Professor Yujiro Kawamata on the occasion of his sixtieth birthday.}

\thanks{supported by JSPS Gran-in-Aid (B) No 22340009, JSPS Grant-in-Aid (S), No 22224001, and by KIAS Scholar Program}

\begin{document}

\maketitle

\begin{abstract}
We prove that the automorphism group of an odd dimensional Calabi-Yau manifold of Picard number two is always a finite group. This makes a sharp contrast to the automorphism groups of K3 surfaces and hyperk\"ahler manifolds and birational automorphism groups, as we shall see. We also clarify the relation between finiteness of the automorphism group (resp. birational automorphism group) and the rationality of the nef cone (resp. movable cone) for a hyperk\"ahler manifold of Picard number two. We will also discuss a similar conjectual relation together with exsistence of rational curve, expected by the cone conjecture, for a Calabi-Yau threefold of Picard number two. 
\end{abstract}

\section{Introduction}

We work over $\mathbf C$. This note is entirely inspired by a question of Doctor Taro Sano to me:
\begin{question}\label{sano}
Let $X$ be a Calabi-Yau threefold of Picard number $2$. 
How the nef cone $\overline{{\rm Amp}}\, (X)$ of $X$ looks like?
\end{question}
Here and hereafter, a {\it Calabi-Yau manifold} is in a wider sense, i.e., a smooth projective manifold $X$ such that ${\mathcal O}_X(K_X) \simeq {\mathcal O}_X$ and $h^1({\mathcal O}_X) = 0$. So a Calabi-Yau manifold in the strict sense, i.e., a smooth simply connected projective manifold $X$ such that ${\mathcal O}_X(K_X) \simeq {\mathcal O}_X$ and $h^0(\Omega_X^k) = 0$ for $0 < k < \dim\, X$ 
and a projective hyperk\"ahler manifold, i.e., a smooth simply connected projective manifold $X$ such that $H^0(\Omega_X^2) = {\mathbf C}\sigma_X$ where $\sigma_X$ is an everywhere non-degenerate $2$-form, are Calabi-Yau manifolds. The Picard group ${\rm Pic}\, (X)$ of a Calabi-Yau manifold $X$ is isomorphic to the N\'eron-Severi group ${\rm NS}\, (X)$. Its rank $\rho(X)$ is the Picard number of $X$. There are many interesting Calabi-Yau manifolds of Picard number $2$ (for instance, \cite{PSS71}, \cite{We88}, \cite{GP01}, \cite{Ku03}, \cite{Ku04} \cite{Ca07}, \cite{Sc11}, \cite{OS01}, \cite{HT01}, \cite{HT09}, \cite{HT09}, \cite{Yo01}, \cite{Yo12}, \cite{Ogr05}, \cite{Sa07}). 
We denote by $(x^n)_X$, 
where $\dim\, X = n$, the top interesection form on ${\rm Pic}\, (X)$. As usual, ${\rm Aut}\,(X)$ (resp. ${\rm Bir}\,(X)$) is the automorphism group (resp. birational automorphism group) of $X$, i.e., the group of biregular self maps (resp. birational self maps) of $X$. Other terminologies below will 
be reviewed in Section 2.

The aim of this note is to prove the following:

\begin{theorem}\label{main1}
Let $X$ be an $n$-dimensional Calabi-Yau manifold with $\rho(X) = 2$. Then: 
\begin{enumerate}
\item 
When $n$ is odd, ${\rm Aut}\, (X)$ is always a finite group. 
\item 
When $n$ is even, ${\rm Aut}\, (X)$ is also a finite group provided that there is no real number $c$ and no real valued quadratic form $q_X(x)$ on 
${\rm NS}(X)$ such that $(x^n)_X = c(q_X(x))^{n/2}$. 
\end{enumerate}
\end{theorem}

We shall prove Theorem (\ref{main1}) in Section 3, which is extremely simple. 

Recall that there is a quadratic form $q_X(x)$ on ${\rm NS}\,(X)$ such that $(x^n)_X = C(q_X(x))^{n/2}$ when $X$ is a hyperk\"ahler manifold. Namely, $q_X(x)$ is the Beauville-Bogomolov form, $C$ is the Fujiki constant, and the relation is the Fujiki relation. Theorem (\ref{main1}) makes a sharp contrast to the following Theorem (\ref{abc}) and Proposition (\ref{main3}):

\begin{theorem}\label{abc}
Let $X$ be a projective hyperk\"ahler manifold with $\rho(X) = 2$. Then: 
\begin{enumerate}
\item 
Either both boundary rays of $\overline{{\rm Amp}}\, (X)$ are rational and ${\rm Aut}\, (X)$ is a finite group or both boundary rays of $\overline{{\rm Amp}}\, (X)$ are irrational and ${\rm Aut}\, (X)$ is an infinite group. Moreover, in the second case, $\overline{{\rm Amp}}\, (X) = \overline{{\rm Mov}}\, (X) = \overline{P(X)}$, the closure of the positive cone with respect to the Beauville-Bogomolov form $q_X(x)$, and ${\rm Bir}\, (X) = {\rm Aut}\, (X)$. 

\item 
Either both boundary rays of the movable cone $\overline{{\rm Mov}}\, (X)$ are rational and ${\rm Bir}\, (X)$ is a finite group or both boundary rays of $\overline{{\rm Mov}}\, (X)$ are irrational and ${\rm Bir}\, (X)$ is an infinite group.
\item 
If the boundary rays of $\overline{{\rm Mov}}\, (X)$ are rational, then the boundary rays of ${\rm Aut}\, (X)$ are also rational. 
 
\item Both cases in both (1) and (2) are realizable in dimension $4$. 
\end{enumerate}
\end{theorem} 

\begin{proposition}\label{main3} There is a Calabi-Yau threefold in the strict sense $X$ with $\rho(X) = 2$ such that 
both boundary rays of $\overline{{\rm Mov}}\, (X)$ are irrational and 
${\rm Bir}\, (X)$ is an infinite group. 
\end{proposition}

Theorem (\ref{abc}) is a generalization of a result of 
Kov\'acs (\cite{Ko94}) 
for K3 surfaces. Theorem (\ref{abc}) (1), (2) are proved in Section 4 by using Markman's solution (\cite{Ma11}) of weak version of the movable cone conjecture, after the global Torelli type results for hyperk\"ahler manifolds due to Huybrechts and Verbitsky (\cite{Hu99}, \cite{Ve09}, \cite{Hu11}). Theorem (\ref{abc}) (3) is proved in Section 5 by using the surjectivity of the period map for hyperk\"ahler manifold due to Huybrechts (\cite{Hu99}) and 
a result of Hassett and Tschinkel (\cite{HT09}). See also 
Corollary (\ref{ex1}) for Hilbert schemes of points on K3 surfaces and generalized Kummer varieties, of Picard number $2$. We prove Proposition (\ref{main3}) by constructing 
an explicit example in Section 6 (Proposition (\ref{ex3})). After posting this note on ArXiv, Professor Kota Yoshioka kindly informed to me that he also constructed hyperk\"ahler manifolds, which are deformation equivalent to generalized Kummer varieties, of Picard number $2$ and of any dimensions with irrational movable cones and those with rational movable cones. They are constructed as the Bogomolov factor $K_{(sH,tH)}(v)$ of the moduli space $M_{(sH,tH)}(v)$ of semi-stable objects with respect to Bridgelad's stabiliy condition $\sigma_{(sH, tH)}$ with Mukai vector $v$ on a polarized abelian surface $(A, H)$ of Picard number $1$ (\cite[Theorem 4.14]{Yo12}). We also note that Hassett and Tschinkel constructed a hyperk\"ahler fourfold $X$ such that $\vert {\rm Aut}\, (X) \vert = 1$ and $\vert {\rm Bir}\, (X) \vert = \infty$, as the Fano scheme of a special cubic fourfold (\cite[Theorem 7.4, Remark 7.6]{HT10}). This result is also kindly 
informed to me by Professor Yuri Tschinkel after posting my note on ArXiv.

Let us return back to the relation of Theorem (\ref{main1}) and Question (\ref{sano}). 
There is a famous conjecture, called the {\it cone conjecture}, due to Kawamata and Morrison (\cite{Mo93}, \cite{Ka97}):
 
\begin{conjecture}\label{kawamata-morrison}
Let $X$ be a Calabi-Yau manifold. Then: 
\begin{enumerate}
\item 
The natural action of ${\rm Aut}\,(X)^*$ on the nef effective cone $\overline{{\rm Amp}}^e\, (X)$ has a finite rational polyhedral fundamental domain. 
\item 
The natural action of ${\rm Bir}\,(X)^*$ on the effective movable cone $\overline{{\rm Mov}}^e\, (X)$ has a finite rational polyhedral fundamental domain. 
\end{enumerate}
Here and hereafter, ${\rm Aut}\,(X)^*$ (resp. ${\rm Bir}\,(X)^*$) is the natural action of ${\rm Aut}\,(X)$ (resp. ${\rm Bir}\,(X)^*$) on the N\'eron-Severi group ${\rm NS}\, (X)$. 
\end{conjecture}

There seems no known example of Calabi-Yau threefold with $\rho(X) = 2$ having 
an irrational boundary of $\overline{{\rm Amp}}\, (X)$. In fact, an affirmative answer to 
Conjecture (\ref{kawamata-morrison}) (1) with Theorem (\ref{main1}) would imply:
\begin{corollary}\label{cor1}

\begin{enumerate}
\item
Let $X$ be an odd dimensional Calabi-Yau manifold with $\rho(X) = 2$. Assume that the cone conjecture (\ref{kawamata-morrison}) (1) is true for this $X$. Then, both boundary rays of $\overline{{\rm Amp}}\, (X)$ are rational. 
\item 
Let $X$ be a Calabi-Yau threefold in the strict sense with $\rho(X) = 2$. 
Assume that the cone conjecture (\ref{kawamata-morrison}) (1) is true for 
this $X$. Then, $X$ contains a rational curve.
\end{enumerate}
\end{corollary}

Corollary (\ref{cor1})(2) is suggested by Professor Paolo Cascini after my 
talk relevant to this work at the conference celebrating the 65-th birthday of Professor Fador Bogolomov at Nantes, France, May 2012. We prove Corollary (\ref{cor1}) in Section 3.

The cone conjecture and existence of rational curve are generally believed to be true at least for Calabi-Yau threefolds in the strict sense. So, Corollary (\ref{cor1}) suggests that the following more explicit form of a question of Doctor Taro Sano could be affirmative: 

\begin{question}\label{sano2}
Let $X$ be a Calabi-Yau threefold in the strict sense with $\rho(X) = 2$. 
\begin{enumerate}
\item
Is $\overline{{\rm Amp}}\, (X)$ always a rational polyhedral cone?
\item 
Is there a rational curve on $X$?
\end{enumerate}
\end{question}

There are Calabi-Yau threefolds $X$ with $\rho(X) = 2$ such that $X$ is an \'etale quotient of a complex torus (\cite[Theorem 01]{OS01}). For such $X$, $\overline{{\rm Amp}}\, (X)$ is always rational polyhedral (\cite[Theorem 01]{OS01}) but there is no rational curve on it. So, our restriction to Calabi-Yau threefold {\it in the strict sense} is harmless for (1) but definitely necessary for (2). 
For relevant work related to (2), see \cite{Wi89}, \cite{Wi92}, \cite{HW92}, \cite{Og93}, \cite{OP98}, \cite{DF11}. 

{\bf Acknowledgement.} I would like to express my thanks to Doctor Taro Sano, 
Professors Fador Bogomolov, Serge Cantat, Paolo Cascini, Yujiro Kawamata, Shinosuke Okawa, Thomas Peternell, De-Qi Zhang for valuable discussions relevant to this work, and to Professors Yuri Tschinkel and Kota Yoshioka for their very important informations. I would like to express my thanks to the referee for very careful reading and several important improvement. Especially, Proposition (\ref{referee}) and its proof are kindly taught me by the referee. 

\section{Preliminaries.}

In this section, we recall the notion of various cones 
from \cite{Ka88}, \cite{Ka97} and a few well known results. For simplicity, $X$ is assumed to be a Calabi-Yau manifold. We consider various cones in ${\rm NS}\, (X)_{{\mathbf R}}$, 
the ${\mathbf R}$-linear extension of 
${\rm Pic}\, (X) \simeq {\rm NS}\, (X)$, with topology of finite dimensional ${\mathbf R}$-linear space and ${\mathbf Z}$-structure given by ${\rm NS}\, (X)$. 
The natural (contravariant) group homomorphism 
$$r : {\rm Bir}\, (X) \to {\rm GL}\, ({\rm NS}\, (X))\,\, ;\,\, g \mapsto g^*$$
is well-defined. This is because 
${\mathcal O}_X(K_X) \simeq {\mathcal O}_X$ 
so that any $g \in {\rm Bir}\, (X)$ acts on $X$ isomorphic in codimension $1$ (see eg. \cite[Page 420]{Ka08}). 

The nef cone $\overline{\rm Amp}\, (X)$ is the closure of the convex hull 
of ample divisor classes. The interior ${\rm Amp}\, (X)$ of $\overline{\rm Amp}\, (X)$ consists of ${\mathbf R}$-ample divisor classes. 
Note that $r({\rm Aut}\, (X))$ induces linear automorphisms of $\overline{\rm Amp}\, (X)$. 

A divisor $D$ is movable if the complete linear system 
$\vert mD \vert$ has no fixed component for some positive integer $m$. All ample divisor classes are movable. The movable cone $\overline{\rm Mov}\, (X)$ is the closure of the convex hull of movable divisor classes. Note that $r({\rm Bir}\, (X))$ induces linear automorphisms of $\overline{\rm Mov}\, (X)$. For a rational point of the interior ${\rm Mov}\, (X)$ 
of $\overline{\rm Mov}\, (X)$, we have the following result due to Kawamata (\cite[Lemma 2.2]{Ka88}):
\begin{proposition}\label{mov}
Any divisor $D$ whose class is in ${\rm Mov}\, (X)$ is movable. 
\end{proposition}

The pseudo-effective cone $\overline{\rm Big}\, (X)$ is the closure of the convex hull of effective divisor classes. $\overline{\rm Big}\, (X)$ is also the closure of the convex hull of big divisor classes and the interior ${\rm Big}\, (X)$ of $\overline{\rm Big}\, (X)$ consists of ${\mathbf R}$-big divisor 
classes (\cite[Lemma 2.2]{Ka88}). 

By definition, $\overline{\rm Amp}\, (X) \subset \overline{\rm Mov}\, (X) \subset \overline{\rm Big}\, (X)$. 

We note that unlike the relative setting, the convex cone $\overline{\rm Big}\, (X)$ is {\it strict} in the sense that 
there contain no straight line through $0$, so that so 
are $\overline{\rm Mov}\, (X)$ and $\overline{\rm Amp}\, (X)$:
\begin{proposition}\label{strict}
$\overline{\rm Big}\, (X)$ is strict. 
\end{proposition}
\begin{proof} Let $n = \dim\, X \ge 2$. Let 
$\pm v \in \overline{\rm Big}\, (X)$. Let $H$ be a very ample divisor 
and $h$ be its class. Then $xh+v$ are ${\mathbf R}$-ample classes for any large real numbers $x$. Considering rational approximation and limit, we obtain
$$((xh+v)^{n-1}.v)_X \ge 0\,\, ,\,\, ((xh+v)^{n-1}.-v)_X \ge 0\,\, ,$$
and therefore $((xh+v)^{n-1}.v)_X = 0$. By expanding the right hand side, we obtain
$$\sum_{k=0}^{n-1}C_k(h^{n-1-k}.v^{k+1})_Xx^{n-1-k} = 0\,\, ,$$
where $C_k > 0$ are binomial coefficients. 
Since $x$ can be any large numbers, it follows that this equality is an equality of polynomials of $x$. In particular,  
$$(h^{n-1}.v)_X = (h^{n-2}.v^2)_X = 0\,\, .$$
Let $H_i$ ($1 \le i \le n-2$) be general elements of $\vert H \vert$. 
Since $H$ is very ample, 
$$S := H_1 \cap H_2 \cap \cdots \cap H_{n-2}$$
is a smooth surface (and $S = X$ when $n=2$). By the formula above, we obtain
$$(h\vert S.v\vert S)_S = ((v \vert S)^2)_S = 0\,\, .$$
We have also $((h\vert S)^2)_S > 0$. Hence, by the Hodge index theorem, 
there is a real number $\alpha$ such that $v \vert S = \alpha h \vert S$ 
in ${\rm NS}\, (S)_{\mathbf R}$. On the other hand, by the Lefschetz hyperplane section theorem, the natural restriction morphism ${\rm NS}\,(X)_{\mathbf R} 
\to {\rm NS}\,(S)_{\mathbf R}$ is injective. Hence $v = \alpha h$ in ${\rm NS}\, (X)_{\mathbf R}$. Substituting this into the equality above, we obtain 
$\alpha (h^n)_X = 0$. Since $(h^n)_X > 0$, it follows that $\alpha = 0$. 
Hence $v = 0$ in ${\rm NS}\, (X)_{\mathbf R}$. This means that $\overline{\rm Big}\, (X)$ is strict.
\end{proof}

We also need the following important result again due to Kawamata \cite[Theorem 5.7]{Ka88}:
\begin{theorem}\label{kawamata}
$\overline{\rm Amp}\, (X) \cap {\rm Big}\,(X)$ is a locally rational polyhedral cone in ${\rm Big}\,(X)$.  
\end{theorem}
When $\rho(X) = 2$, the boundary of $\overline{\rm Amp}\, (X)$ consists of two half lines, say, $\ell_1 = {\mathbf R}_{\ge 0}x_1$ and $\ell_2 ={\mathbf R}_{\ge 0}x_2$, and the boundary of $\overline{\rm Mov}\, (X)$ also consists of two half lines. Theorem (\ref{kawamata}) says that if $x_1$ is big, then 
$\ell_1$ is a rational ray, so that we can rechoose $x_1$ to be rational. 

The effective cone ${\rm Eff}\, (X)$ is the set of points $x \in {\rm NS}\, (X)_{\mathbf R}$ such that there are prime divisors $D_i$ ($1 \le i \le m$) and non-negative real numbers $a_i$ such that $x$ is represented by the class of $\sum_{i=1}^{m} a_iD_i$. ${\rm Eff}\, (X)$ is a convex cone but not closed in 
general. $\overline{\rm Amp}^e\, (X) := \overline{\rm Amp}\, (X) \cap {\rm Eff}\, (X)$ is the effective nef cone and $\overline{\rm Mov}^e\, (X) := \overline{\rm Mov}\, (X) \cap {\rm Eff}\, (X)$ is the effective movable cone. 
Note that $r({\rm Aut}\, (X))$ naturally acts on $\overline{\rm Amp}^e\, (X)$ 
and $r({\rm Bir}\, (X))$ naturally acts on $\overline{\rm Mov}^e\, (X)$. We also note that $g \in {\rm Bir}\, (X)$ is in ${\rm Aut}\, (X)$ if and only if there are ample divisor classes $H_1$ and $H_2$ such that $g^*(H_1) = H_2$ (see eg. 
\cite[Lemma 1.5]{Ka97}). 

The next proposition should be well known even for non-experts:

\begin{proposition}\label{finite}
Let $X$ be a Calabi-Yau manifold and $H$ be an ample line bundle on $X$. 
Let $G$ be a subgroup of ${\rm Bir}\, (X)$ such that $g^*H = H$ in ${\rm Pic}\, (X)$ for all $g \in G$. Then $G \subset {\rm Aut}\, (X)$ and $G$ is a finite group. In particular, if $g^* = id$ for all $g \in G$, then $G$ is a finite subgroup of ${\rm Aut}\, (X)$. 
\end{proposition}

\begin{proof} Since $g^*H = H$ and $H$ is ample, it follows that $g \in {\rm Aut}\, (X)$. Embed $X \subset {\mathbf P}^N$ by $\vert mH \vert$ with large $m$. Let 
$$\tilde{G} := \{g \in {\rm Aut}\, (X)\, \vert\, g^*H = H\}\,\, .$$
Then $G \subset \tilde{G} \subset {\rm Aut}\, (X)$. By the definition of $\tilde{G}$, we have $\tilde{G} \subset {\rm PGL}\, (N)$ and $\tilde{G}$ is the stabilizer of $[X]$ of the natural action of ${\rm PGL}\, (N)$ 
on the Hilbert scheme ${\rm Hilb}_{{\mathbf P}^N}$ of ${\mathbf P}^N$, where 
$[X]$ is the point corresponding to the embedding above. Since the action 
${\rm PGL}\, (N) \times {\rm Hilb}_{{\mathbf P}^N} \to {\rm Hilb}_{{\mathbf P}^N}$ is algebraic, i.e., continuous in the Zariski topology and the point $[X]$ is Zariski closed in ${\rm Hilb}_{{\mathbf P}^N}$, it follows that $\tilde{G}$ is Zariski closed in ${\rm PGL}\, (N)$. Since ${\rm PGL}\, (N)$ is affine noetherian, so is $\tilde{G}$. Since $K_X = 0$ in ${\rm Pic}\, (X)$, it follows that $T_X \simeq \Omega_X^{n-1}$, where $n = \dim\, X$. Denoting by $\overline{A}$ (resp. $A^*$) the complex conjugate (resp. dual) of a complex vector space $A$, we have the following isomorphisms as ${\mathbf C}$-vector spaces
$$H^0(X, T_X) \simeq H^{0}(X, \Omega_X^{n-1}) \simeq \overline{H^{n-1}(X, {\mathcal O}_X)} \simeq \overline{{H^{1}(X, {\mathcal O}_{X}(K_X)}^*} \simeq  \overline{H^{1}(X, {\mathcal O}_X)^*} = 0\, .$$
Here we used the Hodge symmetry, the Serre duality, 
and the fact that ${\mathcal O}_{X}(K_X) \simeq {\mathcal O}_{X}$ 
and $H^1(X, {\mathcal O}_{X}) = 0$. Hence $\dim\, {\rm Aut}\, (X) = 0$. Thus $\dim\, \tilde{G} = 0$ as well. 
Since we already know that $\tilde{G}$ is affine noetherian, this implies that $\tilde{G}$ is a finite set. Since $G \subset \tilde{G}$, the result follows.
\end{proof}

Here we also recall the following 
very important result due to Burnside (\cite[main theorem]{Bu05}):

\begin{theorem}\label{burnside}
Let $G$ be a subgroup of ${\rm GL}(r, {\mathbf C})$. Assume that there is a positive integer $d$ such that $G$ is of exponent $\le d$, i.e., ${\rm ord}\, g \le d$ for all $g \in G$. Here ${\rm ord}\, g$ is the order of $g$ as an element 
of the group $G$. Then $G$ is a finite group.
\end{theorem}

\begin{corollary}\label{finite2}
Let $X$ be a Calabi-Yau manifold. Then: 
\begin{enumerate}
\item 
$[{\rm Bir}\, (X) : {\rm Aut}\, (X)] = [r({\rm Bir}\, (X)) : 
r({\rm Aut}\, (X))]$. 
\item 
Let $G$ be a subgroup of ${\rm Bir}\, (X)$. Then, 
$G$ is finite if and only if there is a positive integer $d$ such that $r(G)$ 
is of exponent $\le d$. 
\end{enumerate}
\end{corollary}

\begin{proof} Since ${\rm Ker}\, r \subset {\rm Aut}\, (X)$ by Proposition (\ref{finite}), the assertion (1) follows. Let us show (2). Recall that ${\rm Ker}\, r$ is a finite group by 
Proposition (\ref{finite}). Assume that 
$r(G)$ is of exponent $\le d$. Then $r(G)$ is a finite group by 
Theorem (\ref{burnside}). Hence $\vert G \vert = 
\vert {\rm Ker}\, r \vert \cdot 
\vert r(G) \vert < \infty$ and we are done for {\it if part}. 
{\it Only if part} is clear.
\end{proof}

\section{Proof of Theorem (\ref{main1}), Corollary (\ref{cor1}).}

In this section, we prove Theorem (\ref{main1}) and Corollary (\ref{cor1}).

From now on, $X$ is an $n$-dimensional Calabi-Yau manifold with $\rho(X) = 2$, 
and $r : {\rm Bir}\, (X) \to {\rm GL}\, ({\rm NS}\,(X))$ is the natural representation as in Section 1.
Since $\rho(X) = 2$, the boundary of $\overline{\rm Amp}\, (X)$ consists of two half lines and the boundary of $\overline{\rm Mov}\, (X)$ also consists of two half lines. We denote 
the two boundary rays of $\overline{{\rm Amp}}\, (X)$ by 
$\ell_1 = {\mathbf R}_{\ge 0}x_1$, $\ell_2 = {\mathbf R}_{\ge 0}x_2$, 
and the two boundary rays of $\overline{{\rm Mov}}\, (X)$ by 
$m_1 = {\mathbf R}_{\ge 0}y_1$, $m_2 = {\mathbf R}_{\ge 0}y_2$. When $\ell_i$ 
(resp. $m_i$) is rational, we always choose $x_i$ (resp. $y_i$) so that $x_i$ 
(resp. $y_i$) is the unique primitive 
integral class on $\ell_i$ (resp. $m_i$). 

Theorem (\ref{main1}) follows from Corollary (\ref {finite2}) 
and the following slightly more general: 

\begin{proposition}\label{finite3} 
\begin{enumerate}
\item 
If there is no quadratic form 
$q_X(x)$ on ${\rm NS}(X)_{\mathbf R}$ such that $(x^n)_X = (q_X(x))^{n/2}$ (when $n$ is even), then $r({\rm Aut}(X))$ is of exponent at most $2$. 
\item 
If at least one of $\ell_i$ is rational, then $r({\rm Aut}\, (X))$ is of exponent at most $2$.
\item 
If at least one of $m_i$ is rational, then $r({\rm Bir}\, (X))$ is of exponent at most $2$.
\end{enumerate}
\end{proposition}

\begin{proof} Let us show (1). Let $g \in {\rm Aut}\, (X)$. Since ${\rm Aut}\, (X)$ acts on $\overline{{\rm Amp}}\, (X)$, which is strictly convex, it follows that 
there are positive real numbers $\alpha > 0$ and $\beta >0$ such that 
$$(g^*)^{2}(x_1) = \alpha x_1\,\, ,\,\, (g^2)^{*}(x_2) = \beta x_2\,\, .$$
Since $g^*$ is defined over ${\mathbf Z}$, it follows that 
${\rm det}\, g^* = \pm 1$ and therefore ${\rm det}\, (g^*)^2 = 1$. Again, since $\overline{{\rm Amp}}\, (X)$ is strictly convex and $\rho(X) =2$, it follows 
that 
$x_1$ and $x_2$ form real basis of ${\rm NS}(X)_{\mathbf R}$. Thus, 
$$\alpha\beta = {\rm det}\, (g^*)^2 = 1\,\, .$$
Let $t,s \in {\mathbf R}$ and consider the element $tx_1 + sx_2$ in ${\rm NS}(X)_{\mathbf R}$. 
Since $g \in {\rm Aut}\, (X)$, we have 
$$(((g^2)^{*}(tx_1 + sx_2))^n)_X = ((tx_1 + sx_2)^n)_X\,\, .$$
Substituting $(g^2)^{*}(tx_1 + sx_2) = \alpha t x_1 + \beta s x_2$ and expanding both sides, we obtain 
$$\sum_{k=0}^{n} C_k \alpha^k \beta^{n-k} (x_1^kx_2^{n-k})_Xt^ks^{n-k} = \sum_{k=0}^{n} C_k (x_1^kx_2^{n-k})_Xt^ks^{n-k}\,\, ,$$
where $C_k > 0$ are binomial coefficients. Since this equality holds for all 
$t, s \in {\mathbf R}$, this is the equality of polynomials of $t$ and $s$. 
Hence 
$$\alpha^k \beta^{n-k} (x_1^kx_2^{n-k})_X = (x_1^kx_2^{n-k})_X\,\, 
{\rm --- (*)}$$ 
for all integers $k$ such that $0 \le k \le n$. 
On the other hand, since $X$ is projective, there are real numbers $s_0$, $t_0$ such that $t_0x_1 + s_0x_2$ is an ample divisor class. Hence by 
$$0 < ((t_0x_1 + s_0x_2)^n)_X = \sum_{k=0}^{n} C_k (x_1^kx_2^{n-k})_Xt_0^ks_0^{n-k}\,\, ,$$
it follows that $I \not= \emptyset$. Here $I$ is the set of integers $k$ such that $0 \le k \le n$ 
and 
$$(x_1^kx_2^{n-k})_X \not= 0\,\, .$$
The set $I$ is independent of the choice of $g$. If $n$ is odd, then $n/2 \not\in I$, and therefore $I \setminus \{n/2\} \not= \emptyset$. 
If $n$ is even and $I = \{n/2\}$, then setting $k = n/2$, we have 
$$((tx_1 + sx_2)^n)_X = C_k(x_1^kx_2^{k})_X(ts)^k = C(ts)^k\,\, .$$
This means that real number $C := C_k(x_1^kx_2^{k})_X$ and a quadratic form 
$$q_X(tx_1 + sx_2) := ts$$ satisfies $(x^n)_X = C(q_X(x))^k$. However, such cases are excluded by the assumption. Therefore under our assumption, 
$$I \setminus \{n/2\} \not= \emptyset\,\, .$$
Thus, there is an integer $k \in I$ such that $k \not= n/2$. For this $k$, we obtain 
$$\alpha^k \beta^{n-k} = 1\,\, $$
from (*). 
On the other hand, $\alpha^{n-k}\beta^{n-k} = 1$ by $\alpha\beta = 1$. Hence
$$\alpha^{2k-n} = 1\,\, .$$
Since $2k - n \not= 0$ by $k \not= n/2$, and $\alpha$ is a positive real number, it follows that 
$\alpha = 1$. Then $\beta = 1$ by $\alpha\beta = 1$. Hence 
$(g^*)^2(x_1) = x_1$ and $(g^*)^2(x_2) = x_2$. Since $x_1$ and $x_2$ form basis of ${\rm NS}(X)_{\mathbf R}$, it follows that $(g^*)^2 = id$. Therefore $r({\rm Aut}\, (X))$ is of exponent $\le 2$. 

Let us show (2). Under the same notation as in the proof of (1), if $x_1$ is primitive and integral, then $(g^*)^{2}(x_1) = x_1$, i.e., $\alpha = 1$. Hence $\beta = 1$. Therefore $(g^*)^2 = id$ for the same 
reason as in the proof of (1). This proves (2). 

Let us show (3). Let $g \in {\rm Bir}\, (X)$. Recall that $\overline{{\rm Mov}}\, (X)$ is strict and $y_1$ and $y_2$ form then basis of ${\rm NS}(X)_{\mathbf R}$. Since ${\rm Bir}\, (X)$ acts on $\overline{{\rm Mov}}\, (X)$, 
it follows that 
there are positive real numbers $\alpha > 0$ and $\beta >0$ such that 
$$(g^*)^{2}(y_1) = \alpha y_1\,\, ,\,\, (g^2)^{*}(y_2) = \alpha y_2\,\, .$$
Since $g^*$ is defined over ${\mathbf Z}$, it follows that 
${\rm det}\, g^* = \pm 1$ and therefore ${\rm det}\, (g^*)^2 = 1$. Thus 
$\alpha \beta = 1$. 
If $y_1$ is primitive and integral, then $(g^*)^{2}(y_1) = y_1$, i.e., $\alpha = 1$. Hence $\beta = 1$. Therefore $(g^*)^2 = id$ for the same 
reason as in (1). This proves (3). 
\end{proof} 

Let us show Corollary (\ref{cor1}). By the assumption, we have a finite rational polyhedral fundamental domain $\Delta$ of the action $({\rm Aut}\, (X))^*$ on $\overline{{\rm Amp}}^e\, (X)$. Since $g^*$ ($g \in {\rm Aut}\, (X)$) are linear and defined over ${\mathbf Z}$, it follows that $g^*\Delta$ are all finite rational polyhedral cones. Since ${\rm Aut}\, (X)$ 
is a finite group by Theorem (\ref{main1}), it follows that the cone
$$\overline{{\rm Amp}}^e\, (X) = \cup_{g \in {\rm Aut}\, (X)} g^*\Delta$$
is also a rational polyhedral cone. Hence, so is its closure 
$\overline{{\rm Amp}}\, (X)$. This means that both boundary rays of $\overline{{\rm Amp}}\, (X)$ are rational. This completes the proof of Corollary (\ref{cor1})(1). 
Let us show (2). It is shown by \cite[Theorem 5.1]{Og93} that a Calabi-Yau threefold $X$ in the strict sense contains a rational curve if there is a non-zero non-ample nef line bundle $L$ such that $(L.c_2(X)) \not= 0$. Since $X$ is a Calabi-Yau threefold in the strict sense, it follows that the linear form $(*, c_2(X))$ is not identically $0$ on ${\rm NS}(X)$ (\cite[Corollary 4.5]{Ko87}). Hence $(*, c_2(X)) \not= 0$ on at least of one boundary ray of $\overline{{\rm Amp}}\, (X)$. Since both boundary rays are rational by (1), the result follows.

\section{Proof of Theorem (\ref{abc}) (1), (2).}

We freely use basic facts on hyperk\"ahler manifolds 
explained in an excellent account by Huybrechts \cite[Part III]{GHJ03}. 

Let $X$ be a projective hyperk\"ahler manifold with $\dim\, X = 2m$. For instance, the Hilbert scheme ${\rm Hilb}^m\, S$ of $m$ points on a projective K3 surface $S$, the generalized Kummer manifold $K_{m}(A)$ of 
an abelian surface $A$, 
and projective manifolds deformation equivalent to them are projective hyperk\"ahler manifold of dimension $2m$ (\cite[21.2]{GHJ03}). We denote by $q_X(x)$ ($x \in H^2(X, {\mathbf Z})$) the Beauville-Bogomolov form (\cite[23.3]{GHJ03}) and by $C = C_X > 0$ the Fujiki constant (\cite[23.4]{GHJ03}). Then 
$(x^{2m})_X = Cq_X(x)^m$ (\cite[23.4]{GHJ03}). The form $q_X(x)$ is of signature $(3,b_2(X)-3)$ on $H^2(X, {\mathbf Z})$ and $q_X(x) \vert {\rm NS}\, (X)$ is of signature $(1, \rho(X) -1)$ (\cite[23.3, 26.4]{GHJ03}). Moreover, $q_X(h) \ge 0$ for $h \in \overline{{\rm Amp}}\, (X)$ and $q_X(g^*x) = q_X(x)$ for any $g \in {\rm Bir}\, (X)$ and any $x \in {\rm NS}\, (X)_{{\mathbf R}}$ (\cite[26.4, 27.1]{GHJ03}). We denote by $P(X)$ the positive cone of $X$, i.e., the connected component of 
$$\{x \in {\rm NS}\,(X)_{\mathbf R} \vert q_X(x) > 0\}$$
containing the ample classes. $\overline{P(X)}$ is the closure of $P(X)$. 

The following proposition should be well known:
\begin{proposition}\label{hk1}
\begin{enumerate}
\item 
$\overline{{\rm Mov}}\, (X) \subset \overline{P(X)} \subset 
\overline{{\rm Big}}\, (X)$.
\item
$P(X) \subset {\rm Big}\, (X)$.
\end{enumerate}
\end{proposition}

\begin{proof} The fundamental inclusion $P(X) \subset {\rm Big}\, (X)$ is the projectivity criterion due to Huybrechts (\cite[23.3, 26.4]{GHJ03}). In fact, once we accept the projective criterion (\cite[Erratum, Theorem 2]{Hu99}, \cite[26.4]{GHJ03}), for a given primitive divisor class $h \in P(X)$, the general deformation of $X$ keeping $h$ being integral $(1,1)$ class, is a projective hyperk\"ahler manifold with the class $h$ being the primitive ample generator of 
the Picard group, therefore the upper-semicontinuity theorem implies the bigness of $h$ on $X$ (\cite[Proof of Theorem 4.6]{Hu99}). The inclusions 
$P(X) \subset 
{\rm Big}\,(X)$, $\overline{P(X)} \subset \overline{{\rm Big}}\, (X)$ follow from this. If $x$ is a movable divisor class, then there are $k >0$, $D_1, D_2 \in \vert kx \vert$ such that $S := D_1 \cap D_2$ is a subscheme of pure codimension $2$ (unless it is empty). By the definition of $q_X(x)$, up to positive multiples, we have  
$$q_X(x) = \int_X x^2(\sigma_X\overline{\sigma}_X)^{2m-2} = \int_S (\sigma_X\overline{\sigma}_X)^{2m-2} \ge 0\,\, .$$
This implies $\overline{{\rm Mov}}\, (X) \subset \overline{P(X)}$. 
\end{proof}

We will use the following slightly weaker version of a very important result due to Markman (\cite[Theorem 6.25]{Ma11}; for the closedness of $\Delta$ in 
${\rm NS}\, (X)_{\mathbf R}$, see also proof there or its source \cite[Page 511]{St85}):

\begin{theorem}\label{hk2} 
There is a finite rational polyhedral cone 
$\Delta \subset \overline{{\rm Mov}}\, (X)$ 
such that
$$\overline{{\rm Mov}}\, (X) = \overline{\cup_{g \in {\rm Bir}\,(X)} g^*\Delta}\,\, ,\,\, (g^*\Delta)^o \cap (\Delta)^o = \emptyset$$
unless $g^* = id$ on ${\rm NS}\, (X)$. Here $A^o$ is the interior of $A$.
\end{theorem}

From now on, $X$ is a projective hyperk\"ahler manifold with $\rho(X) = 2$. 
As in Section 3, We denote 
the two boundary rays of $\overline{{\rm Amp}}\, (X)$ by 
$\ell_1 = {\mathbf R}_{\ge 0}x_1$, $\ell_2 = {\mathbf R}_{\ge 0}x_2$, 
and the two boundary rays of $\overline{{\rm Mov}}\, (X)$ by 
$m_1 = {\mathbf R}_{\ge 0}y_1$, $m_2 = {\mathbf R}_{\ge 0}y_2$. 

\begin{lemma}\label{hk3}
Assume that $m_1 = {\mathbf R}_{\ge 0}y_1$ is rational. 
Then $m_2 = {\mathbf R}_{\ge 0}y_2$ is also rational 
and ${\rm Bir}\, (X)$ is a finite group.
\end{lemma}

\begin{proof} By Proposition (\ref{finite3})(3), 
${\rm Bir}\, (X)$ is a finite group. Then $\cup_{g \in {\rm Bir}\,(X)} g^*\Delta$ is closed. Hence by Theorem (\ref{hk2}), 
$$\overline{\rm Mov}\, (X) = \overline{\cup_{g \in {\rm Bir}\,(X)} g^* \Delta} 
= \cup_{g \in {\rm Bir}\,(X)} g^* \Delta\,\, ,$$ 
is a rational polyhedral domain. Thus $m_2 = {\mathbf R}_{\ge 0}x_2$ is 
rational.
\end{proof}

\begin{lemma}\label{hk4}
Assume that $m_1 = {\mathbf R}_{\ge 0}y_1$ is irrational. 
Then $m_2 = {\mathbf R}_{\ge 0}y_2$ is also irrational 
and ${\rm Bir}\, (X)$ is an infinite group.
\end{lemma}

\begin{proof}
By Lemma (\ref{hk3}), $m_2 = {\mathbf R}_{\ge 0}y_2$ is irrational. 
If ${\rm Bir}\, (X)$ would be finite, then $\cup_{g \in {\rm Bir}\,(X)} g^*\Delta$ is closed, and by Theorem (\ref{hk2}), 
$$\overline{\rm Mov}\, (X) = \overline{\cup_{g \in {\rm Bir}\,(X)} g^* \Delta} 
= \cup_{g \in {\rm Bir}\,(X)} g^* \Delta\,\, ,$$
a contradiction, because any boundary rays of the right hand side is rational.
Hence ${\rm Bir}\, (X)$ must be an infinite group. 
\end{proof}

\begin{lemma}\label{hk5}
Assume that $\ell_1 = {\mathbf R}_{\ge 0}x_1$ is rational. 
Then $\ell_2 = {\mathbf R}_{\ge 0}x_2$ is also rational 
and ${\rm Aut}\, (X)$ is a finite group.
\end{lemma}

\begin{proof} By Proposition (\ref{finite3})(3), ${\rm Aut}\, (X)$ is a finite group. If $q_X(x_2) > 0$, then $x_2 \in {\rm Big}\, (X)$ by 
Proposition (\ref{hk1})(2). Therefore $\ell_2$ is rational by Theorem (\ref{kawamata}) and we are done.
So, we may assume that $q_X(x_2) = 0$. Then, by Proposition (\ref{hk1})(1) and 
by $\overline{{\rm Amp}}\, (X) \subset \overline{{\rm Mov}}\, (X)$, the half line $\ell_2 = {\mathbf R}_{\ge 0}x_2$ is also a boundary ray of $\overline{\rm Mov}\, (X)$. 

{\it From now, assuming to the contrary that $\ell_2$ is irrational, we shall derive a contradiction.} By Theorem (\ref{hk2}), we have 
$$\overline{\rm Mov}\, (X) = \overline{\cup_{g \in {\rm Bir}\,(X)} g^*\Delta}\,\, , \,\, (g^*\Delta)^o \cap (\Delta)^o = \emptyset\,\, ,$$
unless $g^* = id$. Let $p_i$ ($i = 1$, $2$) be the two boundary rays of $\Delta$. Since $(g_1^*\Delta)^o \cap (g_2^*\Delta)^o = \emptyset$ for $g_1^* \not= g_2^*$, it follows that $g^*p_1$ and $g^*p_2$ ($g \in {\rm Bir}\, (X)$) decompose the (irrational) polyhedoral cone $\overline{\rm Mov}\, (X)$ into infinite (rational) chambers $g^*\Delta$. Since $\ell_2$ is irrational, there is then a sequence $\{g_k \in {\rm Bir}\, (X)\}_{k \ge 0}$ such that $g_k^* \not= g_{k'}^*$ for $k \not= k'$ and 
$$g_{k}^*\Delta \to \ell_2\,\, ,\,\, (k \to \infty)\,\, ,$$
where the limit is taken inside the compact set 
$$({\rm NS}\, (X)_{\mathbf R} \setminus \{0\})/{\mathbf R}_{> 0}^{\times}\,\, .$$
 Since $\ell_2$ is a boundary ray of both 
$$\overline{{\rm Amp}}\, (X) \subset \overline{{\rm Mov}}\, (X)\,\, ,$$ 
there is then $k_0$ such that 
$$g_{k}^*\Delta \subset \overline{{\rm Amp}}\, (X)$$
for all $k \ge k_0$. However, then for any interior integral point $a$ of $\Delta$, we have 
$$b := g_{k_{0}}^*(a) \in {\rm Amp}\, (X)\,\, ,\,\, 
g_k^*(a) \in {\rm Amp}\, (X)$$ 
and therefore $(g_{k_{0}}^{-1}g_k)^{*}(b) \in {\rm Amp}\, (X)$. Hence, 
$g_{k_{0}}^{-1}g_k \in {\rm Aut}\, (X)$, and therefore 
$g_k \in g_{k_{0}}{\rm Aut}\, (X)$ for all $k \ge k_0$. Here $\{g_k \vert k \ge k_0\}$ is an infinite set but ${\rm Aut}\, (X)$ is a finite set, a 
contradiction. Hence $\ell_2$ must be rational.
\end{proof}

\begin{lemma}\label{hk6}
Assume that $\ell_1 = {\mathbf R}_{\ge 0}x_1$ is irrational. 
Then $\ell_2$ is also irrational. Moreover, 
$\overline{{\rm Amp}}\, (X) = \overline{{\rm Mov}}\, (X) = \overline{P(X)}$ and ${\rm Bir}\, (X) = {\rm Aut}\, (X)$ is an infinite group.
\end{lemma}

\begin{proof} If $\ell_2$ would be rational, then $\ell_1$ would be rational by Lemma (\ref{hk5}). Hence $\ell_2$ must be irrational. For the same reason as in the first part of the proof of Lemma (\ref{hk5}), it follows that 
$$q_X(x_1) = q_X(x_2) = 0\,\, .$$ 
Combining this with Proposition (\ref{hk1})(1), we obtain the required equalities of the cones. 
Since the boundary rays of $\overline{{\rm Mov}}\, (X) = \overline{{\rm Amp}}\, (X)$ are irrational, it follows that ${\rm Bir}\, (X)$ is an infinite group by Lemma (\ref{hk4}). Moreover, for $g \in {\rm Bir}\, (X)$, we have $g^*(\overline{{\rm Mov}}\, (X)) = \overline{{\rm Mov}}\, (X)$ and therefore $g^*(\overline{{\rm Amp}}\, (X)) = \overline{{\rm Amp}}\, (X)$ by $\overline{{\rm Mov}}\, (X) = \overline{{\rm Amp}}\, (X)$. Thus $g \in {\rm Aut}\, (X)$. Hence 
${\rm Bir}\, (X) \subset {\rm Aut}\, (X)$. This completes the proof.
\end{proof}

Theorem (\ref{abc})(2) follows from Lemma (\ref{hk3}) and Lemma (\ref{hk4}) 
and Theorem (\ref{abc})(1) follows from Lemma (\ref{hk5}) and Lemma (\ref{hk6}). 
\begin{remark}\label{sarti} Boissiere and Sarti proved a remarkable result that ${\rm Bir}\, (X)$ is finitely generated for any projective hyperk\"ahler manifold $X$ (\cite[Theorem 2]{BS09}). On the other hand, it seems unknown if ${\rm Aut}\, (X)$ is finitely generated or not (\cite[Question 1]{BS09}). When $\rho(X) = 2$, Theorem (\ref{abc}) says that ${\rm Aut}\, (X)$ is either a finite group or equal to ${\rm Bir}\, (X)$. So, ${\rm Aut}\, (X)$ is always finitely generated when $\rho(X) = 2$.  
\end{remark}

\section{Proof of Theorem (\ref{abc}) (3).}

In this section, we shall prove Theorem (\ref{abc}) (3). 

We begin with the following general result which is kindly taught us by the referee:

\begin{proposition}\label{referee}
Let $X$ be a projective hyperk\"ahler manifolds of Picard number $2$. 
If $X$ contains an effective divisor $E$ such that $q_X(E) < 0$, then 
${\rm Bir}\, (X)$ is finite and $\overline{\rm Mov}\, (X)$ is rational.
\end{proposition}

\begin{proof} By replacing $E$ by one of its irreducible components, we may assme that $E$ is irreducible and reduced. The ray 
$$\{x \in P(X)\, \vert\, (x, [E]) <0\}$$
is then a boundary ray of ${\rm Mov}\, (X)$ by \cite[Lemmas 6.20, 6.22]{Ma11}. 
This is clearly a rational ray. Therefore, the result follows from 
Theorem (\ref{abc})(2). 
\end{proof}

Standard series of projective hyperk\"ahler manifolds of Picard number $2$ 
are examples of the first case in Theorem (\ref{abc})(1)(2):

\begin{corollary}\label{ex1}
Let $m$ be an integer such that $m \ge 2$. 
\begin{enumerate}
\item
Let $S$ be a projective K3 surface with $\rho(S) = 1$. 
Then $\rho({\rm Hilb}^m\, S) =2$, the boundary rays of $\overline{\rm Amp}\, ({\rm Hilb}^m\, S)$ and $\overline{\rm Mov}\, ({\rm Hilb}^m\, S)$ are rational 
and ${\rm Aut}\, ({\rm Hilb}^m\, S)$ and ${\rm Bir}\, ({\rm Hilb}^m\, S)$ are finite groups. 

\item 
Let $A$ be an abelian surface with $\rho(A) = 1$. 
Then $\rho(K_{m}(A)) =2$, the boundary rays of $\overline{\rm Amp}\, (K_{m}(A))$ and $\overline{\rm Mov}\, (K_{m}(A))$ are rational 
and ${\rm Aut}\, (K_{m}(A))$ and ${\rm Bir}\, (K_{m}(A))$ are finite groups. 
\end{enumerate}
\end{corollary}

\begin{proof} The exceptional divisor $E \subset {\rm Hilb}^m\, S$ of the Hilbert-Chow morphism satisfies $q_{{\rm Hilb}^m\, S}(E) < 0$. So, Proposition (\ref{referee}) implies (1) and similarly (2). 
\end{proof}

We recall a lattice isomorphism 
$$(H^2({\rm Hilb}^2\, S, {\mathbf Z}), q_{{\rm Hilb}^2\, S}(x)) \, \simeq \Lambda := U^{\oplus 3} \oplus E_8(-1)^{\oplus 2} \oplus \langle -2 \rangle$$
for the Hilbert scheme ${\rm Hilb}^2\, S$ of a K3 surface $S$ (\cite[Page 187]{GHJ03}). Here $U = {\mathbf Z}e \oplus {\mathbf Z}f$ is a lattice of rank $2$ with bilinear form given by 
$$(e,e)_U = (f, f)_U = 0\,\, ,\,\, (e, f)_U = 1$$
and $E_8(-1)$ is a negative definite even unimodular lattice of rank $8$. 
So, $\Lambda$ is a lattice of signature $(3, 20)$. 
We also consider the lattice 
$$L := {\mathbf Z} h_1 \oplus {\mathbf Z} h_2$$ 
of rank $2$ whose bilinear form is defined by
$$(h_1, h_1)_L = (h_2, h_2)_L = 4\,\, ,\,\, (h_1, h_2)_L = 8\,\, .$$
\begin{proposition}\label{ex2}
Let $S$ be a K3 surface.
\begin{enumerate}
\item
There is a projective hyperk\"ahler manifold such that 
$X$ is deformation equivalent to ${\rm Hilb}^2\, S$ and ${\rm NS}\, (X) \simeq L$ as lattices (where the lattice structure on ${\rm NS}\, (X)$ is the one given by the Beauville-Bogomolov form of $X$).
\item 
For this $X$, both boundary rays of $\overline{{\rm Amp}}\, (X)$ are 
irrational. In particular, $X$ is an example of the second case in 
Theorem (\ref{abc})(1)(2).
\end{enumerate}
\end{proposition}

\begin{proof} For $xh_1 + yh_2 \in L$ ($x, y \in {\mathbf Z}$), we have
$$(xh_1 + yh_2, xh_1 + yh_2)_L = 4(x^2 + 4xy + y^2) = 
4((x+2y)^2 - 3y^2)\,\, .$$
In particular, $L$ is an even lattice of signature $(1,1)$. 
This lattice $L$ admits a primitive embedding 
into $U^{\oplus 3}$ as lattices, given by  
$$h_1 \mapsto 2e_1 + e_2 +2f_2\,\, ,\,\, h_2 \mapsto 4f_1 + e_3 + 2f_3$$
where $e_i, f_i$ is the standard basis of the $i$-th $U$ ($i = 1$, $2$, $3$). 
This embedding naturally define a primitive embedding $L \subset \Lambda$ 
via the standard embedding $U^{\oplus 3} \subset \Lambda$. 
Then 
$$L^{\perp} := \{x \in \Lambda \vert (x, L) = 0\}$$ 
is a lattice of singnature $(2, 19)$. 
Let us choose a positive real $2$-plane $P$ in $(L^{\perp})_{\mathbf R}$ which is not in any rational hyperplanes of $(L^{\perp})_{\mathbf R}$. 
This is possible, because positive real two planes form an open subset in the 
real Grassmanian manifold ${\rm Gr}\, (2, (L^{\perp})_{\mathbf R})$. 
Let $\langle x, y \rangle$ be an orthonormal basis of $P$ and set 
$$\sigma := x + \sqrt{-1} y \in \Lambda_{\mathbf C}\,\, .$$ 
Then, by the choice of $P$ and the primitivity of $L$, we obtain that
$$(\sigma, \sigma) = 0\,\, ,\,\, (\sigma, \overline{\sigma}) > 0\,\, ,\,\, (\sigma)^{\perp} \cap \Lambda = L\,\, .$$
Then, by the surjectivity of the period map due to Huybrechts (\cite[25.4]{GHJ03}), there is a hyperk\"ahler manifold $X$ which is deformation equivalent to ${\rm Hilb}^2\, S$ and ${\rm NS}\, (X) \simeq L$. Since $L$ is of signature $(1,1)$, $X$ is projective 
by the projectivity criterion (\cite[26.4]{GHJ03}). Moreover for any $v \in {\rm NS}\, (X) \simeq L$, we have $4 \vert q_X(v)$ from the explicit formula in $L$ above. Hence there is no $v \in {\rm NS}\, (X)$ such that $q_X(v) = -2$ or $q_X(v) = -10$. 
Thus by a result of Hassett and Tschinkel (\cite[Theorem 23]{HT09}), we obtain 
$$\overline{{\rm Amp}}\, (X) = \overline{P (X)}\,\, .$$ 
Moreover, there is no $v \in {\rm NS}\, (X)$ such that $q_X(v) = 0$ other than $0$ again by the explicit formula in $L$ above. Hence both boundary rays of $\overline{P}\, (X) = \overline{{\rm Amp}}\, (X)$ are irrational. Now, by Theorem (\ref{abc})(2), this $X$ is an example of the second cases of Theorem 
(\ref{abc}) (1), (2). 
\end{proof} 

Theorem (\ref{abc})(3) now follows from Proposition (\ref{ex1}) applied for $m=2$ and Proposition (\ref{ex2}). 

\section{Proof of Proposition (\ref{main3}).}

In this section, we shall prove Proposition (\ref{main3}). This will follow from Proposition (\ref{ex3}).

In what follows, for $X \subset {\mathbf P}^{3} \times {\mathbf P}^{3}$, 
we use the following symbols:

${\mathbf P} := {\mathbf P}^{3} \times {\mathbf P}^{3}$;

$\iota : X \to {\mathbf P}$, the natural inclusion morphism;

$P_i : {\mathbf P} \to {\mathbf P}^{3}$, the natural $i$-th projection ($i = 1$, $2$);  

$p_i := P_i \circ \iota : X \to {\mathbf P}^{3}$, the natural $i$-th projection from $X$ ($i = 1$, $2$);

$L_i := {\mathcal O}_{{\mathbf P}^{3}}(1)$, the hyperplane bundle of ${\mathbf P}^{3}$ ($i = 1$, $2$);

$H_i := P_i^*L_i$, the line bundle which is the pull back of the hyperplane bundle $L_i$ to  ${\mathbf P}$ ($i = 1$, $2$);; 

$h_i := p_i^*L_i = \iota^*H_i$, the line bundle which is the pull back of the hyperplane bundle $L_i$ to $X$ ($i = 1$, $2$); 

\begin{proposition}\label{ex3}
Let $X$ be a general complete intersection of three hypersurfaces of bidegree $(1,1)$, $(1,1)$ and $(2,2)$ in ${\mathbf P}^3 \times {\mathbf P}^3$. Then, 
$X$ is a smooth Calabi-Yau threefold such that 
$${\rm NS}(X) \simeq {\rm Pic}\, X = {\mathbf Z}h_1 \oplus {\mathbf Z}h_2\,\, ,\,\, 
\overline{{\rm Amp}}\, (X) = {\mathbf R}_{\ge 0} h_1 +  {\mathbf R}_{\ge o} h_2\,\, ,$$ 
$$\overline{{\rm Mov}}\, (X) = {\mathbf R}_{\ge 0} (-h_1 + (3+2\sqrt{2})h_2)  
+ {\mathbf R}_{\ge 0} ((3+2\sqrt{2})h_1 - h_2)\,\, ,$$ 
$${\rm Bir}\, (X) = \langle {\rm Aut}\, (X), \tau_1, \tau_2 \rangle$$ 
where $\tau_1$ and $\tau_2$ are birational involutions of $X$ 
and $\tau_1^* \tau_2^*$ is of infinite order.  
\end{proposition}

\begin{proof} The fact that $X$ is a smooth Calabi-Yau threefold with ${\rm Pic}\, X = {\mathbf Z}h_1 \oplus {\mathbf Z}h_2$ follows from the Bertini theorem, adjunction formula and the Lefschetz hyperplane section theorem. The projection $p_i$ ($i = 1$, $2$) are of degree $2$ by the shape of the equation above. 
We note that both $p_i$ are {\it not} finite. More precisely, they are small contractions, i.e. contract at least one curves but contract no divisor. This can be seen as follows. Let
$$F_1 := x_0g_0(y) + x_1g_1(y) + x_2g_2(y) + x_3g_3(y) = 0\, ,$$
$$F_2 := x_0h_0(y) + x_1h_1(y) + x_2h_2(y) + x_3h_3(y) = 0\, ,$$
$$F_3 := \sum_{0\le i \le j \le 3}x_ix_jk_{ij}(y) = 0\, ,$$
be the defining equations of $X$. 
Here $([x_0:x_1:x_2:x_3], y = [y_0:y_1:y_2:y_3])$ is the homogeneous coordinate of ${\mathbf P}^3 \times {\mathbf P}^3$, for each $i = 0, 1, 2, 3$
$$g_i(y) = a_{i0}y_0 + a_{i1}y_1 + a_{i2}y_2 + a_{i3}y_3\, ,$$
$$h_i(y) = b_{i0}y_0 + b_{i1}y_1 + b_{i2}y_2 + b_{i3}y_3\, ,$$ 
and $k_{ij}(y)$ are homogeneous polynomial of degree $2$. 
By the genericity assumption, matrices $A = (a_{ij})$ 
and $B = (b_{ij})$ are general $4 \times 4$ complex matrices, in particular, {\it they are invertible, $BA^{-1}$ has four distinct eigenvalues and therefore $BA^{-1}$ is diagonalizable}. 
Let $V$ be the complete intersection $4$-fold defined by the first 
two equations
$$F_1 := x_0g_0(y) + x_1g_1(y) + x_2g_2(y) + x_3g_3(y) = 0\, ,$$
$$F_2 := x_0h_0(y) + x_1h_1(y) + x_2h_2(y) + x_3h_3(y) = 0\, .$$
Consider the first projection $p : V \rightarrow {\mathbf P}^3$. Then for each $P = [p_0:p_1:p_2:p_3] \in {\mathbf P}^3$, the fiber $p^{-1}(P)$ is defined by 
the equations
$$p_0g_0(y) + p_1g_1(y) + p_2g_2(y) + p_3g_3(y) = 0\,\, ,\,\, 
p_0h_0(y) + p_1h_1(y) + p_2h_2(y) + p_3h_3(y) = 0$$ 
in ${\mathbf P}^3$. So, $p^{-1}(P)$ is isomorphic to either a line ${\mathbf P}^1$ or a plane ${\mathbf P}^2$ in ${\mathbf P}^3$, and it is a plane 
exactly when the defining equations of $p^{-1}(P)$ are proportional, i.e., exactly when
$$(p_0, p_1, p_2, p_3)(\alpha A - \beta B) = 0\,\, ,\,\, {\rm i.e.}\,\, ,\,\, (p_0, p_1, p_2, p_3)(\alpha I - \beta BA^{-1}) = 0$$
for some $[\alpha : \beta] \in {\mathbf P}^1$. Here $I$ is the identity matrix. Then such $[\alpha : \beta]$ satisfies ${\rm det}\, (\alpha I - \beta BA^{-1}) = 0$. This is a homogneous 
polynomial of 
degree $4$ with respect to $[\alpha : \beta]$. Hence there are exactly four 
{\it distinct} solutions $[\alpha : \beta]$ and for each solution, ${\rm rank}\, (\alpha I - \beta BA^{-1}) = 3$, by the genericity assumption of $A$ and $B$, which we make explicit above. So, the first projection $p : V \rightarrow {\mathbf P}^3$ has exactly four ${\mathbf P}^2$ as fibers. Since the third equation $F_3 = 0$ of $X$, is also general, it follows that ${\mathbf P}^2 \cap X = {\mathbf P}^2 \cap (F_3 = 0)$ is a curve for each of these $4$ fibers ${\mathbf P}^2$. These curves are contracted by $p_1$, as they are in fibers of $p$. Note that for other fibers of $p$, which are ${\mathbf P}^1$, we have a priori $\dim ({\mathbf P}^1 \cap X = {\mathbf P}^1 \cap (F_3 = 0)) \le 1$. Hence $p_1$ is a small contraction. For the same reason, $p_2$ is also a small contraction.  

Hence $h_1$ and $h_2$ are both semi-ample but none of them is ample. 
From this, we obtain $\overline{{\rm Amp}}\, (X) = {\mathbf R}_{\ge 0} h_1 +  {\mathbf R}_{\ge o} h_2$. 

Let $\tau_i : X \cdots \rightarrow X$ be the covering involution with 
respect to 
$p_i$ ($i =1, 2$). Let 
$$\nu_i : X \rightarrow \overline{X}_i$$
be the Stein factorization of $p_i$ and $\overline{p}_i : \overline{X}_i \rightarrow {\mathbf P}^3$ be the induced morphism. The covering involution $\tau_i$ induces the {\it biregular} involution $\overline{\tau}_i$ of $\overline{X}_i$ 
over ${\mathbf P}^3$. This is because the Stein factorization is unique in the rational funcion field of $X$.

\begin{lemma}\label{matrix}
With respect to the basis $\langle h_1, h_2 \rangle$ of ${\rm NS}\, (X)$, 
$$\tau_1^* = \left(\begin{array}{rr}
1 & 6\\
0 & -1
\end{array} \right)\,\, ,\,\, \tau_2^* = \left(\begin{array}{rr}
-1 & 0\\
6 & 1
\end{array} \right)\,\, .$$
\end{lemma}
\begin{proof} The proof here is similar to \cite{Si91}. By definition of $\tau_1$, we have $\tau_1^*h_1 = h_1$. 
We can write $(p_1)_*h_2 = aL_1$, where $a$ is an integer. Here the pushforward is the pushfoward as divisors. Then
$$a = ((p_1)_*h_2.L_1^2)_{{\mathbf P}^3} = ((p_1)_*p_2^*L_2.L_1^2)_{{\mathbf P}^3} = (p_2^*L_2.p_1^*L_1^2)_X = (h_2.h_1^2)_X$$
by the projection formula. Since $X = 2(H_1+H_2)^3$ as cycles in $\mathbf P$, it follows that
$$(h_2.h_1^2)_X = (H_1.H_2^2.2(H_1+H_2)^3)_{\mathbf P} = 6\,\,.$$
Combining these two equalities, we obtain $a = 6$. Hence
$$h_2 + \tau_1^*h_2 = p_1^*(p_1)_*h_2 = p_1^*(6L_1) = 6h_1\,\, .$$
Then $\tau_1^*h_2 = -h_2 + 6h_1$. This together with $\tau_1^*h_1 = h_1$ 
proves the result for $\tau_1^*$. The proof for $\tau_2^*$ is identical. 
\end{proof}
\begin{lemma}\label{flop2}
$\overline{\tau}_i^{-1} \circ \nu_i : X_i^{+} := X \rightarrow \overline{X}_i$ is the flop of $\nu_i : X \rightarrow \overline{X}_i$ and $(\overline{\tau}_i^{-1} \circ \nu_i)^{-1} \circ \nu_i = \tau_i$ as birational automorphisms of $X$.
\end{lemma}
\begin{proof} The proof here is similar to \cite{Og11}. 
The second statement follows from $\overline{\tau}_i \circ \nu_i = \nu_i \circ \tau_i$. 
The relative Picard number $\rho(X/\overline{X}_1)$ 
is $1$ because $\rho(X) = 2$ and $\nu_1$ is a non-trivial projective 
contraction. By Lemma (\ref{matrix}), we have
$$\tau_1^*h_2 = -h_2 + 6h_1\,\, .$$
Thus $\tau_1^*h_2$ is relatively anti-ample for $\overline{\tau}_1^{-1} \circ \nu_1 : X \rightarrow \overline{X}_1$, 
while $h_2$ is relatively ample for $\nu_1$. 
Since $K_X = 0$, the map $\overline{\tau}_1^{-1} \circ \nu_1 : X_1^{+} := X \rightarrow \overline{X}_1$ is then the flop of $\nu_1 : X \rightarrow \overline{X}_1$. 
The proof for $i=2$ is identical.
\end{proof}

\begin{lemma}\label{bir} 
${\rm Bir}\, (X) = {\rm Aut}\, (X) \cdot \langle \tau_1, \tau_2 \rangle$.
\end{lemma}

\begin{proof} The proof here is also similar to \cite{Og11}. Recall that any flopping contraction of a Calabi-Yau manifold is given by a codimension one face of $\overline{{\rm Amp}}\, (X)$ up to automorphisms of $X$ 
(\cite[Theorem (5.7)]{Ka88}). Since there is no codimension one face of $\overline{{\rm Amp}}\, (X)$ other than ${\mathbf R}_{\ge 0}h_i$ ($i = 1, 2$), it follows that there is no flop other than $\tau_i : X \cdots\to X$ ($i = 1, 2$) up to ${\rm Aut}\,(X)$. On the other hand, by a result of Kawamata (\cite[Theorem 1]{Ka08}), any birational map between minimal models is decomposed into finitely many flops up to automorphisms of the target variety. Thus any $\varphi \in {\rm Bir}\, (X)$ is decomposed into a finite sequence of flops $\tau_i$ 
and an automorphism of $X$ at the last stage. This proves the result.
\end{proof}

\begin{lemma}\label{elem} Let $n$ be an integer. Then, with respect to the basis $\langle h_1, h_2 \rangle$ of ${\rm NS}\, (X)$ (resp. $\langle (2\sqrt{2}+3)h_1 - h_2, -h_1 + (2\sqrt{2}+3)h_2 \rangle$ of ${\rm NS}\, (X)_{\mathbf R}$), 
$$(\tau_1^*\tau_2^*)^n =  \left(\begin{array}{rr}
35 & 6\\
-6 & -1
\end{array} \right)^n\,\, , \,\, {\rm resp.}\,\, ,\,\,  
(\tau_1^*\tau_2^*)^n =  \left(\begin{array}{rr}
(17 + 12\sqrt{2})^n & 0\\
0 & (17 - 12\sqrt{2})^n
\end{array} \right)\,\, .$$
Here $(2\sqrt{2}+3)h_1 - h_2$ (resp. $-h_1 + (2\sqrt{2}+3)h_2$) is an eigen vector of $\tau_1^*\tau_2^*$, corresponding to the eigenvalue $17 + 12\sqrt{2} > 1$ (resp. $17 - 12\sqrt{2} = 1/(17 + 12\sqrt{2})$) of $\tau_1^*\tau_2^*$. 
In particular, $\tau_1^*\tau_2^*$ 
is of infinite order.
\end{lemma}
\begin{proof} Results follow from standard, concrete calculation in $2 \times 2$ matrices. For the last satatement, we use $17 + 12\sqrt{2} > 1$. 
\end{proof}
In what follows, we put
$$M := ({\mathbf R}_{> 0} (-h_1 + (3+2\sqrt{2})h_2  + {\mathbf R}_{> 0} ((3+2\sqrt{2})h_1 - h_2)) \cup \{0\}\,\, ,$$
$$A := \overline{{\rm Amp}}\, (X) = {\mathbf R}_{\ge 0} h_1  + {\mathbf R}_{\ge 0} h_2\,\, .$$
\begin{lemma}\label{comp} 
${\rm Bir}\, (X)^*A = M$. 
\end{lemma}
\begin{proof} 
If $g \in {\rm Aut}\, X$ acts nontrivially on ${\rm NS}\, (X)$, 
then $g^*h_1 = h_2$ and $g^*h_2 = h_1$ and $g^{*}M = M$ by the shape of $M$. 
So, it suffices to show that 
$\langle \tau_1^*, \tau_2^* \rangle(A)  = M$.
Since $\tau_i$ ($i = 1, 2$) are involutions, each element of $\langle \tau_1^*, \tau_2^* \rangle$ is of the following forms with 
$n \in {\mathbf Z}$:
$$(\tau_1^*\tau_2^*)^n\,\, ,\,\, \tau_2^*(\tau_1^*\tau_2^*)^n\,\, .$$
(For this, we also note that $\tau_1^* = \tau_2^*(\tau_2^*\tau_1^*)$ and $\tau_2^*\tau_1^* = (\tau_2^*)^{-1}(\tau_1^*)^{-1} = (\tau_1^*\tau_2^*)^{-1}$.) Now the result follows from Lemma (\ref{elem}) together with elementary calculation in $2 \times 2$ matrices, based on Lemmas (\ref{elem}), (\ref{matrix}).  
\end{proof} 
\begin{lemma}\label{fin}
$\overline{{\rm Mov}}\, (X) = \overline{M}$. 
\end{lemma}
\begin{proof} By Lemma (\ref{comp}), we have $M \subset \overline{{\rm Mov}}\, (X)$. Hence $\overline{M} \subset \overline{{\rm Mov}}\, (X)$. Let ${\rm Mov}\, (X)({\mathbf Q})$ be the set of rational point in the interior ${\rm Mov}\, (X)$ of $\overline{{\rm Mov}}\, (X)$. Let $d \in {\rm Mov}\, (X)({\mathbf Q})$. Then by Proposition (\ref{mov}), there is a positive integer 
$m$ and an effective movable divisor $D$ such that $md = [D]$. The pair $(X, \epsilon D)$ is klt for small positive rational number $\epsilon$. Note that $K_X + \epsilon D = \epsilon D$ by $K_X = 0$ and $\dim\, X = 3$. So, if $D$ 
is not nef, then we can run log minimal model program for the pair $(X, \epsilon D)$ to make $D$ nef. By the shape of $A = \overline{{\rm Amp}}\, (X)$, the first step in this program is either one of 
$\tau_i : X \cdots \to X$ ($i = 1$, $2$) 
and the manifold $X$ remains the same. Hence, so are every other step in the program. Hence there is $g \in \langle \tau_1, \tau_2 \rangle$ such that 
$g^*[D] \in A$, whence $g^*d \in A$. If $D$ is nef, we can choose $g = id$. 
Thus, $d \in (g^*)^{-1}(A) \subset M$ in any case and therefore 
${\rm Mov}\, (X)({\mathbf Q}) \subset M$. Since $\overline{{\rm Mov}\, (X)({\mathbf Q})} = \overline{{\rm Mov}}\, (X)$ (just by general topology), it follows that 
$\overline{{\rm Mov}}\, (X) \subset \overline{M}$. 
This completes the proof. 
\end{proof} 
Now we complete the proof of Proposition (\ref{ex3}).
\end{proof}

\end{document}